\documentclass[11pt,a4paper]{article}
\usepackage[utf8]{inputenc}
\usepackage{amsmath}
\usepackage{amsthm}
\usepackage{amsfonts}
\usepackage{amssymb}
\author{Mark Shusterman}
\title{Ranks of subgroups in \\ boundedly generated groups}
\usepackage{mathtools}
\usepackage{graphicx}
\usepackage{hyperref}
\newtheorem{theorem}{Theorem}[section]
\newtheorem{lemma}[theorem]{Lemma}
\newtheorem{proposition}[theorem]{Proposition}
\newtheorem{corollary}[theorem]{Corollary}
\numberwithin{equation}{section}

\newcommand{\lemref}[1]{\hyperref[#1]{Lemma \ref*{#1}}}
\newcommand{\thmref}[1]{\hyperref[#1]{Theorem \ref*{#1}}}
\newcommand{\propref}[1]{\hyperref[#1]{Proposition \ref*{#1}}}
\newcommand{\corref}[1]{\hyperref[#1]{Corollary \ref*{#1}}}

\makeatletter
\def\moverlay{\mathpalette\mov@rlay}
\def\mov@rlay#1#2{\leavevmode\vtop{%
   \baselineskip\z@skip \lineskiplimit-\maxdimen
   \ialign{\hfil$\m@th#1##$\hfil\cr#2\crcr}}}
\newcommand{\charfusion}[3][\mathord]{
    #1{\ifx#1\mathop\vphantom{#2}\fi
        \mathpalette\mov@rlay{#2\cr#3}
      }
    \ifx#1\mathop\expandafter\displaylimits\fi}
\makeatother

\makeatletter
\newcommand*{\defeq}{\mathrel{\rlap{%
                     \raisebox{0.27ex}{$\m@th\cdot$}}%
                     \raisebox{-0.27ex}{$\m@th\cdot$}}%
                     =}
\makeatother

\begin{document}

\maketitle

\abstract{We show that an infinite residually finite boundedly generated group has an infinite chain of finite index subgroups with ranks uniformly bounded, and give (sublinear) upper bounds on the ranks of arbitrary finite index subgroups of boundedly generated groups (examples which come close to achieving these bounds are presented). This proves a strong form of a conjecture of Abert, Jaikin-Zapirain, and Nikolov which asserts that the rank gradient of infinite boundedly generated residually finite groups is $0$. Furthermore, our first result establishes a variant of a conjecture of Lubotzky on the ranks of finite index subgroups of special linear groups over the integers, and is analogous to a  result of Pyber and Segal for solvable groups.

\section{Introduction}

Given a positive integer $m$, 
we say that a group $G$ is $m$-\textbf{boundedly generated} if there exist $g_1, \dots, g_m \in G$ 
such that the following equality of sets holds:
\begin{equation} \label{DefBGEq}
G = \langle g_1 \rangle \cdots \langle g_m \rangle.
\end{equation}
A group $G$ is said to be \textbf{boundedly generated} if it is $m$-boundedly generated for some positive integer $m$. The prototype of a boundedly generated group is $\mathrm{SL}_3(\mathbb{Z})$ (see \cite[Section 4.1]{BHV} for a proof), and more generally, 
for any integer $n \geq 3$ and a number field $K$, 
the group $\mathrm{SL}_n(O_K)$ is known to be boundedly generated (see \cite{AM}, \cite{CK1}, \cite{CK2}, \cite{CKP}, and \cite{W}). On the other hand, it is well known that $\mathrm{SL}_2(\mathbb{Z})$ is not boundedly generated (see \url{en.wikipedia.org/wiki/Boundedly_generated_group} for various proofs) and our work is a generalization of this fact. Still, $\mathrm{SL}_2(A)$ is boundedly generated for certain rings of integers $A$ (e.g. for those having infinitely many units), and bounded generation is a feature of many other arithmetic groups (see \cite{Bar}, \cite{DV}, \cite{ER1}, \cite{ER2}, \cite{H},  \cite{LM}, \cite{M}, \cite{SS}, \cite{T1}, \cite{T2}, and \cite{Z}). Bounded generation of arithmetic groups has been studied with relation to the congruence subgroup property (see \cite{Lub}, \cite{PR}, and \cite{R}), and in connection with Kazhdan's property (T) (see \cite[Section 4]{BHV}, and \cite{S}). The notion of bounded generation is also of independent interest as can be seen from \cite{ALP}, \cite{Mura}, \cite{NS}, \cite{PS},  and \cite{Su}.

Our first result is:

\begin{theorem} \label{FirstResThm}

Let $m$ be a positive integer, 
let $G$ be an $m$-boundedly generated group,
and let $L \leq G$ be a finite index subgroup. 
Then there exists a finite index subgroup $U \leq L$ which can be generated by $m$ elements.

\end{theorem}

The theorem asserts that the finite index subgroups of rank (smallest cardinality of a generating set) at most $m$ are abundant - they form a basis for the open neighborhoods of $1$ with respect to the profinite topology on $G$ (the coarsest topology for which any homomorphism to a finite group is continuous).

Another source of examples of boundedly generated groups is the family of finitely generated solvable groups of finite subgroup rank (i.e. those for which there exists a finite bound on the ranks of all finitely generated subgroups - see \cite{PS}). A converse is proved by Pyber and Segal (see \cite[Corollary 1.5]{PS}): A residually finite (i.e. a group which is Hausdorff with respect to its profinite topology) boundedly generated solvable group has finite subgroup rank. \thmref{FirstResThm} can be viewed as an analogue of their result for groups which are not necessarily solvable. We note that a conclusion as strong as the one in the result of Pyber and Segal (a uniform bound on the ranks of all finitely generated subgroups) is not achievable in the nonsolvable case, as can be seen from the example given in \ref{SecLB}.

In \cite[Section 4, Problem 1]{Lub0} Lubotzky conjectured that for any integer $n \geq 3$ and every finite index subgroup $L \leq \mathrm{SL}_n(\mathbb{Z})$, there exists a finite index subgroup $U \leq L$ of rank $2$. It has been shown in \cite{LR} that $\mathrm{SL}_3(\mathbb{Z})$ contains an infinite descending chain of finite index subgroups of rank $2$. Moreover, \cite{SV} establishes the conjecture for a larger class of groups than $\mathrm{SL}_n(\mathbb{Z})$ while allowing a rank bound greater than $2$ (namely, $3$). Similarly, \thmref{FirstResThm} considers a larger family of groups ($m$-boundedly generated) and replaces $2$ by the constant $m$. Another feature of \thmref{FirstResThm} is that it gives a bound on the index $[G : U]$ (see \lemref{MainLem}).

Recall that the rank gradient of a finitely generated group $G$ is defined to be:
\begin{equation} \label{DefRGEq}
\inf_{H} \frac{\mathrm{rank}(H)-1}{[G : H]} 
\end{equation}
where $H$ runs over all subgroups of finite index in $G$. This notion has been introduced by Lackenby in his study of manifolds (see \cite{Lack}) and became a subject of active research (see \cite{AGN}, \cite{AJN}, \cite{AN}, \cite{AG},  \cite{E}, \cite{Gi}, \cite{GK}, \cite{KN}, \cite{Lack2}, \cite{O}, \cite{P}, and \cite{Sch}). It is interesting to know for which groups the rank gradient is positive. Some examples are free nonabelian groups, $\mathrm{SL}_2(\mathbb{Z})$, surface groups, free products (see \cite[Proposition 3.2]{Lack}), and certain torsion groups (see \cite{O}, and \cite{Sch}). Nonexamples include finite groups, abelian groups, residually finite solvable groups, residually finite amenable groups (see \cite{AJN}), residually finite ascending HNN extensions (see \cite{Lack}), automorphism groups of free nonabelian groups (see \cite{KN}), and many arithmetic groups (e.g. $\mathrm{PSL}_2(\mathbb{Z}[i])$, $\mathrm{PSL}_2(\mathbb{Z}[e^{\frac{2\pi i}{3}}])$ - see \cite[Examples A,B]{AN}, and also higher rank nonuniform arithmetic groups such as $\mathrm{SL}_3(\mathbb{Z})$ - see \cite{SV}). These examples led Abert, Jaikin-Zapirain, and Nikolov to conjecture that the rank gradient of infinite residually finite boundedly generated groups is zero (see \cite{AJN}). In this direction, they use Luck's approximation theorem to prove a weaker result under the additional assumption of the existence of a finite presentation (see \cite[Proposition 11]{AJN}). \thmref{FirstResThm} establishes this conjecture in a strong form: the theorem can be applied inductively to construct subgroups of growing finite index and bounded rank. It follows at once (see \eqref{DefRGEq}) that the rank gradient vanishes. 

Our next result confirms an even stronger form of the aforementioned conjecture.

\begin{theorem} \label{SecResThm}

Let $m$ be a positive integer, 
let $G$ be an $m$-boundedly generated group,
and let $L \leq G$ be a subgroup of finite index $n$.
Then
\begin{equation} \label{GenBoundEq}
\mathrm{rank}(L) \leq 2m^2\sqrt{n\log(n)} + m.
\end{equation}
If in addition $L \lhd G$, then
\begin{equation} \label{NormalBoundEq}
\mathrm{rank}(L) \leq m\log_2(n) + m.
\end{equation} 
\end{theorem} 

That is, the rank of subgroups grows sublinearly with the index (the rate of growth of the rank as a function of the index has been studied for other groups in \cite{AG} and \cite{GK}). This implies that the rank gradient of any chain (see \cite{AJN} or \cite{AN} for the precise definition) in a boundedly generated group is zero. In particular, the rank gradient for boundedly generated groups is independent of the choice of a Farber chain (see \cite{AN} for definitions). This independence is suspected to hold for all countable groups, and is an important special case of the fixed price problem (see \cite{AN},\cite{G}, and \cite{KM}). Note that fixed price has already been established for some infinite boundedly generated groups such as $\mathrm{SL}_2(A)$ for rings of integers $A$ with $|A^{*}|$ infinite (see \cite{KM}).

The bounds given in \thmref{SecResThm} are new even for arithmetic boundedly generated groups. In \ref{SecLB} we give an example which comes close to achieving the bound given in \eqref{NormalBoundEq}.

\section{Preliminaries}

Let us recall some definitions and facts needed in the sequel. The material presented in this section is quite standard and straightforward.

Let $G$ be a group, and let $g,x \in G$. We set $g^x \defeq x^{-1}gx$ and note that this defines a right action of $G$ on itself by automorphisms. For a subset $S \subseteq G$ we similarly put $S^x \defeq x^{-1}Sx$. For example, we have
\begin{equation} \label{ExpSubEq}
\langle g \rangle^x = x^{-1} \langle g \rangle x =  \langle x^{-1}gx \rangle = \langle g^x \rangle.
\end{equation}

We define the rank of a group $G$ to be the smallest cardinality of a generating set of $G$, and abbreviate our notation by $d(G) \defeq \mathrm{rank}(G)$. For a group epimorphism $\pi \colon G \to H$ we clearly have
\begin{equation} \label{RankQuotInEq}
d(H) \leq d(G).
\end{equation}

\begin{proposition} \label{RankBoundProp}

Let $L$ be a group,
and let $U \leq L$ be a finite index subgroup.
Then $d(L) \leq d(U) + \log_2([L : U])$.

\end{proposition}

\begin{proof}

We induct on $[L : U]$, 
observing that for the base case (when the index is $1$) we have equality. 

Assume thus that $U \lneq L$ and let $M$ be a maximal proper subgroup of $L$ containing $U$. 
Take some $x \in L \setminus M$ and set
\begin{equation} \label{UltraMaxDefEq}
K \defeq \langle M \cup \{x\} \rangle.
\end{equation}
We have $M \lneq K \leq L$, so the maximality of $M$ implies that $K = L$. We infer that 
\begin{equation} \label{LMP1Eq}
d(L) = d(K) \stackrel{\mathclap{\eqref{UltraMaxDefEq}}}{\leq} d(M) + 1
\end{equation}
and induction tells us that 
\begin{equation} \label{BoundMEq}
d(M) \leq d(U) + \log_2([M : U]).
\end{equation}
Since $[L : M] \geq 2$ it follows that 
\begin{equation}
\begin{split}
d(L) &\stackrel{\mathclap{\eqref{LMP1Eq}}}{\leq} d(M) + 1 \stackrel{\ref{BoundMEq}}{\leq} d(U) + \log_2([M : U]) + 1 \\
&\leq d(U) + \log_2([M : U]) + \log_2([L : M]) \\
&= d(U) + \log_2([L : M][M : U]) 
= d(U) + \log_2([L : U]).
\end{split}
\end{equation}

\end{proof}

For a group $G$, and a subgroup $L \leq G$ we take 
\begin{equation} \label{ExpDefEq}
\mathrm{exp}(G,L)
\end{equation}
to be the least positive integer $n$ such that $a^n \in L$ for each $a \in G$. If such an integer does not exist, we set $\mathrm{exp}(G,L) \defeq \infty$. If we take $N$ to be the intersection of all the conjugates of $L$ in $G$, then $N \lhd G$ and the exponent of the group $G/N$ equals $\mathrm{exp}(G,L)$. Let us now show that $\mathrm{exp}(G,L)$ is finite in case that $[G : L]$ is.

\begin{proposition} \label{LagrangeProp}

Let $G$ be a group, 
let $a \in G$, 
and let $L \leq G$ be a subgroup of finite index $n$.
Then there is some integer $1 \leq r \leq n$ such that $a^r \in L$.

\end{proposition}

\begin{proof}

Define $f \colon \{0, 1, \dots, n\} \to G/L$ by $f(i) = a^{i}L.$ Since the cardinality of the domain of $f$ is larger than that of its codomain, $f$ is not injective. That is, there exist $0 \leq i < j \leq n$ such that $a^jL = a^iL$, which means that $a^{j - i} \in L$. We can thus take $r \defeq j - i$. 
 
\end{proof}

\begin{corollary} \label{ExpCor}

Let $G$ be a group, 
and let $L \leq G$ be a subgroup of finite index $n$.
Then $\mathrm{exp}(G,L) \leq \mathrm{lcm}(1, \dots, n).$
If in addition $L \lhd G$, then $\mathrm{exp}(G,L) \leq n$.

\end{corollary}

\begin{proof}

Let $a \in G$. By \propref{LagrangeProp}, there exists an $1 \leq r \leq n$ with $a^r \in L$. Since $r \ | \ \mathrm{lcm}(1, \dots, n)$ it follows that $a^{\mathrm{lcm}(1, \dots, n)} \in L$. 
If \mbox{$L \lhd G$}, then $\mathrm{exp}(G,L)$ is just the exponent of $G/L$ which is at most $|G/L| =  n$.

\end{proof}

In order to estimate the rate of growth of $\mathrm{lcm}(1, \dots, n)$ as a function of $n$, we set $k \defeq \lfloor \frac{n-1}{2} \rfloor$ and note that
$$\mathrm{lcm}(1, \dots, n) \int_{0}^1 (x(1-x))^k \leq 
\mathrm{lcm}(1, \dots, n) \int_{0}^1 4^{-k} 
= \frac{\mathrm{lcm}(1, \dots, n)}{4^k} $$ 
where the left hand side is a positive integer since the integral of a polynomial with integer coefficients and degree at most $n-1$ is a sum of fractions with denominators not exceeding $n$. 
It follows that the right hand side is at least $1$ so
$\mathrm{lcm}(1, \dots, n) \geq 4^k \geq 4^{\frac{n-3}{2}} = \frac{1}{8}2^n$
showing us that the first bound in \corref{ExpCor} is at least exponential in $n$. 
In case that the group is boundedly generated, a better bound exists.
For that, recall that Landau's function $g(n)$ associates to each $n \in \mathbb{N}$ the largest integer which is an order of an element in $S_n$. Alternatively, 
\begin{equation} \label{LandauDefEq}
g(n) \defeq \max \ \{|\langle \sigma \rangle| :  \sigma \in S_n\}.
\end{equation}

\begin{proposition} \label{LandauProp}

Let $m$ be a positive integer, 
let $G$ be an $m$-boundedly generated group,
and let $L \leq G$ be a subgroup of finite index $n$. 
Then 
\begin{equation} \label{BGboundEq}
\mathrm{exp}(G,L) \leq g(n)^m.
\end{equation}

\end{proposition}

\begin{proof}

Let $x \in G$.
By \eqref{DefBGEq} there exist $g_1, \dots, g_m \in G$ such that
\begin{equation} \label{BGEq}
G = \langle g_1 \rangle \cdots \langle g_m \rangle.
\end{equation}
The action of $G$ on $G/L$ induces a homomorphism 
$\varphi \colon G \to S_n$. 
Set
\begin{equation} \label{DefhEq}
H \defeq \varphi(G), \quad h_i \defeq \varphi(g_i) \quad (1 \leq i \leq m)
\end{equation} 
so that
\begin{equation} \label{HstEq}
H \stackrel{\ref{DefhEq}}{=} \varphi(G) 
\stackrel{\ref{BGEq}}{=}
\langle \varphi(g_1) \rangle \cdots \langle \varphi(g_m) \rangle 
\stackrel{\ref{DefhEq}}{=} \langle h_1 \rangle \cdots \langle h_m \rangle.
\end{equation}
Since $h_1, \dots, h_m \in S_n$, it follows that 
\begin{equation} \label{OrdhEq}
|H| \stackrel{\ref{HstEq}}{=} |\langle h_1 \rangle \cdots \langle h_m \rangle|
\leq |\langle h_1 \rangle| \cdots |\langle h_m \rangle|
\stackrel{\ref{LandauDefEq}}{\leq} g(n)^m.
\end{equation}
As $\varphi(x) \in H$, it follows that 
$\varphi(x^{|H|}) = \varphi(x)^{|H|} = 1$
so $x^{|H|} \in \mathrm{Ker}(\varphi)$ and thus $x^{|H|} \in L$. By \eqref{OrdhEq}, 
$\mathrm{exp}(G,L) \leq |H| \leq g(n)^m$.
\end{proof}

In order to see that our new bound \eqref{BGboundEq} in \propref{LandauProp} is indeed better than the previous exponential bound from \corref{ExpCor}, we recall that (see \cite{Ma})
\begin{equation} \label{LanLiminEq}
\log(g(n)) \leq 1.05314\sqrt{n\log(n)}
\end{equation}
which implies that $g(n)$ grows subexponentially.

\section{Main lemma}

Before passing to our main lemma (\lemref{MainLem}), let us discuss two ingredients form its proof. 

The first one is just the use of several identities which hold in every group (see \eqref{LongEq}). A similar equality is used in some proofs of the fact that a finite index subgroup of a boundedly generated group is itself boundedly generated (see \cite[Lemma 2.1 (i)]{NS} and especially equation 2.2 therein which shortly states an assertion similar to \eqref{PasEq}, or \cite{T1}).

The second ingredient is a beautiful lemma of Bernhard Hermann Neumann (see \cite[Lemma 4.1]{Neumann}) stating that if a group is covered by finitely many cosets of subgroups, then the index of at least one of these subgroups does not exceed the number of cosets in the covering. This lemma has already been used with connection to bounded generation (see \cite[Theorem 2.2, Lemma 3.2]{NS}).

We are now ready for the main lemma which gives \thmref{FirstResThm}. Given a positive integer $n$, we use $[n]$ as a shorthand for $\{0, 1, \dots, n-1\}$. 

\begin{lemma} \label{MainLem}

Let $m$ be a positive integer, 
let $G$ be an $m$-boundedly generated group,
and let $L \leq G$ be a finite index subgroup. 
Then there exists a subgroup $U \leq L$ such that $d(U) \leq m$, and $[G : U] \leq \mathrm{exp}(G,L)^m$.

\end{lemma}

\begin{proof}

By \eqref{DefBGEq}, there exist $g_1, \dots, g_m \in G$ such that 
\begin{equation} \label{BG2Eq}
G = \langle g_1 \rangle \cdots \langle g_m \rangle.
\end{equation}
Set $n \defeq \mathrm{exp}(G,L)$ (which is finite by \corref{ExpCor}), and note that for each $g \in G$ we have
\begin{equation} \label{ModEq}
\langle g \rangle = \bigcup_{k \in [n]} g^{k}\langle g^n \rangle
\end{equation} 
and thus 
\begin{equation} \label{CovEq}
G \stackrel{\ref{BG2Eq}}{=} \prod_{i=1}^m \langle g_i \rangle 
\stackrel{\ref{ModEq}}{=} \prod_{i=1}^m \bigcup_{k_i \in [n]} g_i^{k_i}\langle g_i^n \rangle 
= \bigcup_{\vec{k} \in [n]^m} \  \prod_{i=1}^m g_i^{k_i}\langle g_i^n \rangle.
\end{equation}
For $1 \leq r \leq m, \ 1 \leq \ell \leq m + 1,$ and $\vec{k} \in [n]^m$ set 
\begin{equation} \label{Peq}
P_{\ell,r}^{\vec{k}} \defeq  \prod_{s = \ell}^r g_s^{k_s} \in G
\end{equation}
and observe that for $1 \leq r \leq m$ we have
\begin{equation} \label{RecEq}
P_{r+1,r}^{\vec{k}} \stackrel{\ref{Peq}}{=} 1
\end{equation}
while for $1 \leq r \leq m - 1, \ 1 \leq \ell \leq r + 1$ we have
\begin{equation} \label{Rec2Eq}
P_{\ell,r}^{\vec{k}}g_{r+1}^{k_{r+1}} \stackrel{\ref{Peq}}{=} P_{\ell,r+1}^{\vec{k}}.
\end{equation}
For $1 \leq r \leq m$ and $\vec{k} \in [n]^m$ define the following subsets of $G$:
\begin{equation} \label{SubEq}
S_r^{\vec{k}} \defeq 
P_{1,r}^{\vec{k}}
\Bigg[ \prod_{\ell=1}^r \langle (g_\ell^{P_{\ell+1,r}^{\vec{k}}})^n \rangle \Bigg] 
\Bigg[ \prod_{t=r+1}^m g_t^{k_t} \langle g_{t}^n \rangle \Bigg] \subseteq G.
\end{equation}
We show now that for $1 \leq r \leq m-1$ we have $S_{r}^{\vec{k}}  = S_{r+1}^{\vec{k}}$:

\begin{equation} \label{LongEq}
\begin{split}
S_{r}^{\vec{k}} &\stackrel{\ref{SubEq}}{=} 
P_{1,r}^{\vec{k}}
\Bigg[ \prod_{\ell=1}^r \langle (g_\ell^{P_{\ell+1,r}^{\vec{k}}})^n \rangle \Bigg] 
\Bigg[ \prod_{t=r+1}^m g_t^{k_t} \langle g_{t}^n \rangle \Bigg] \\
&= P_{1,r}^{\vec{k}}
g_{r+1}^{k_{r+1}}g_{r+1}^{-k_{r+1}} 
\Bigg[ \prod_{\ell=1}^r \langle (g_\ell^{P_{\ell+1,r}^{\vec{k}}})^n \rangle \Bigg] 
g_{r+1}^{k_{r+1}} \langle g_{r+1}^n  \rangle 
\Bigg[ \prod_{t=r+2}^m g_t^{k_t} \langle g_{t}^n \rangle \Bigg] \\
&\overset{\ref{Rec2Eq}}{\underset{\ref{RecEq}}{=}} P_{1,r+1}^{\vec{k}}
\Bigg[ \prod_{\ell=1}^r \langle (g_\ell^{P_{\ell+1,r}^{\vec{k}}})^n \rangle \Bigg]^{g_{r+1}^{k_{r+1}}} 
\langle (g_{r+1}^{P_{r+2,r+1}^{\vec{k}}})^n  \rangle 
\Bigg[ \prod_{t=r+2}^m g_t^{k_t} \langle g_{t}^n \rangle \Bigg] \\
&\stackrel{\ref{ExpSubEq}}{=} P_{1,r+1}^{\vec{k}}
\Bigg[ \prod_{\ell=1}^r \langle (g_\ell^{P_{\ell+1,r}^{\vec{k}}g_{r+1}^{k_{r+1}}})^n \rangle  \Bigg]
\langle (g_{r+1}^{P_{r+2,r+1}^{\vec{k}}})^n  \rangle 
\Bigg[ \prod_{t=r+2}^m g_t^{k_t} \langle g_{t}^n \rangle \Bigg] \\
&\stackrel{\ref{Rec2Eq}}{=} P_{1,r+1}^{\vec{k}}
\Bigg[ \prod_{\ell=1}^r \langle (g_\ell^{P_{\ell+1,r+1}^{\vec{k}}})^n \rangle  \Bigg]
\langle (g_{r+1}^{P_{r+2,r+1}^{\vec{k}}})^n  \rangle 
\Bigg[ \prod_{t=r+2}^m g_t^{k_t} \langle g_{t}^n \rangle \Bigg] \\
&= P_{1,r+1}^{\vec{k}}
\Bigg[ \prod_{\ell=1}^{r+1} \langle (g_\ell^{P_{\ell+1,r+1}^{\vec{k}}})^n \rangle  \Bigg]
\Bigg[ \prod_{t=r+2}^m g_t^{k_t} \langle g_{t}^n \rangle \Bigg] \stackrel{\ref{SubEq}}{=} S_{r+1}^{\vec{k}}.
\end{split}
\end{equation}
It follows that $S_{r}^{\vec{k}}$ is independent of $r$, so in particular, $S_{1}^{\vec{k}} = S_{m}^{\vec{k}}$. This equality, and \ref{SubEq} tell us that
\begin{equation} \label{PasEq}
\prod_{i=1}^m g_i^{k_i}\langle g_i^n \rangle = 
P_{1,m}^{\vec{k}}
\Bigg[ \prod_{\ell=1}^m \langle (g_\ell^{P_{\ell+1,m}^{\vec{k}}})^n \rangle \Bigg].
\end{equation}
Taking the union in \ref{PasEq} over all values of $\vec{k} \in [n]^m$ we get:
\begin{equation}
\begin{split} 
G &\stackrel{\ref{CovEq}}{=} 
\bigcup_{\vec{k} \in [n]^m} \ \prod_{i=1}^m g_i^{k_i}\langle g_i^n \rangle \\
&\stackrel{\ref{PasEq}}{=} \bigcup_{\vec{k} \in [n]^m} 
P_{1,m}^{\vec{k}}
\Bigg[ \prod_{\ell=1}^m \langle (g_\ell^{P_{\ell+1,m}^{\vec{k}}})^n \rangle \Bigg] \\
&\subseteq
\bigcup_{\vec{k} \in [n]^m} 
P_{1,m}^{\vec{k}}
\Bigg \langle (g_\ell^{P_{\ell+1,m}^{\vec{k}}})^n \Bigg \rangle_{\ell=1}^m
\end{split}
\end{equation} 
so $G$ is covered by finitely many cosets of subgroups. By \cite[Lemma 4.1]{Neumann}, the index in $G$ of one of these subgroups, say $U$, is at most the number of cosets in the covering which is $|[n]^m| = n^m$. By \eqref{ExpDefEq}, the $n$-th power of an element in $G$ must belong to $L$, so all generators of $U$ are in $L$. It follows that $U \leq L$, and visibly $d(U) \leq m$.

\end{proof}

Let us now deduce \thmref{SecResThm}.

\begin{corollary} \label{MainCor}

Let $m$ be a positive integer, 
let $G$ be an $m$-boundedly generated group,
and let $L \leq G$ be a subgroup of finite index $n$.
Then
\begin{equation} \label{MainCor1stGoal}
d(L) \leq 2m^2\sqrt{n\log(n)} + m.
\end{equation}
If in addition $L \lhd G$, then $d(L) \leq m\log_2(n) + m$. 

\end{corollary}

\begin{proof}

By \lemref{MainLem}, there exists a subgroup $U \leq L$ of rank at most $m$ and 
\begin{equation} \label{IndexBoundedEq}
[L : U] \leq [G : U] 
\stackrel{\ref{MainLem}}{\leq} \mathrm{exp}(G,L)^m 
\stackrel{\ref{LandauProp}}{\leq} (g(n)^m)^m = g(n)^{m^2}.
\end{equation}
Therefore, 
\begin{equation}
\begin{split}
d(L) &\stackrel{\ref{RankBoundProp}}{\leq} d(U) + \log_2([L : U]) \leq m + \log_2([L : U]) \\
&\stackrel{\ref{IndexBoundedEq}}{\leq} m + \log_2(g(n)^{m^2}) = m + m^2\log_2(g(n)) \\ 
&= m + m^2\log(g(n))\log_2(e) \\
&\stackrel{\ref{LanLiminEq}}{\leq} m + m^2 \cdot 1.1 \cdot \sqrt{n\log(n)} \cdot 1.5 \\
&\leq 2m^2\sqrt{n\log(n)} + m. 
\end{split}
\end{equation}
If $L \lhd G$ then 
\begin{equation} \label{NormalIndexBoundedEq}
[L : U] \stackrel{\ref{IndexBoundedEq}}{\leq} \mathrm{exp}(G,L)^m \stackrel{\ref{ExpCor}}{\leq} n^m
\end{equation} so 
\begin{equation}
\begin{split}
d(L) &\stackrel{\ref{RankBoundProp}}{\leq} d(U) + \log_2([L : U]) \leq m + \log_2([L : U]) \\
&\stackrel{\ref{NormalIndexBoundedEq}}{\leq} m + \log_2(n^m) \leq m\log_2(n) + m.
\end{split}
\end{equation}

\end{proof}

\section{Lower bound} \label{SecLB}

We show that $G \defeq \mathrm{SL}_4(\mathbb{Z})$ comes asymptotically close to achieving the bound \eqref{NormalBoundEq} from \thmref{SecResThm}. As we are about to see, a choice of a different arithmetic group could have been made as well. Our example is an elaboration of the nonexistence example given in \cite{SuV}.

Let $r$ be a positive integer, and let $p_1, \dots, p_r$ be the first $r$ odd primes. For $1 \leq i \leq r$ let $T_i  \lhd \mathrm{SL}_4(\mathbb{F}_{p_i})$ be the subgroup containing only $\pm I_4$. Set 
\begin{equation} \label{LowerDefsEq}
H_r \defeq \prod_{i=1}^r \mathrm{SL}_4(\mathbb{F}_{p_i}) \quad Z_r \defeq \prod_{i=1}^r T_i \lhd H_r
\end{equation} 
and recall that by the strong approximation property (see \cite{K}), the natural projection $\pi_r \colon G \to H_r$ is a surjection. Put $L_r \defeq \pi_r^{-1}(Z_r)$ and note that
\begin{equation} \label{LowerRankBoundEq}
d(L_r) \stackrel{\ref{RankQuotInEq}}{\geq} d(Z_r) \stackrel{\ref{LowerDefsEq}}{=} d((\mathbb{Z}/2\mathbb{Z})^r) = r. 
\end{equation}
Furthermore,
\begin{equation} \label{LowerIndexBoundEq}
\begin{split}
[G : L_r] &\stackrel{\ref{LowerDefsEq}}{=} [H_r : Z_r] \leq 
|H_r| \stackrel{\ref{LowerDefsEq}}{\leq} \prod_{i=1}^r |\mathrm{SL}_4(\mathbb{F}_{p_i})| 
\leq \prod_{i=1}^r p_i^{4^2} \\
&= \Bigg[ \prod_{i=1}^r p_i \Bigg]^{16} \leq 
\Big((r+1)^{2r+2}\Big)^{16} \leq (2r)^{48r}.
\end{split}
\end{equation}
On the other hand,
\begin{equation} \label{LowerIndexBound2Eq}
\begin{split}
[G : L_r] &\stackrel{\ref{LowerDefsEq}}{=} [H_r : Z_r] = 
\frac{|H_r|}{|Z_r|} \stackrel{\ref{LowerDefsEq}}{=} 2^{-r}\prod_{i=1}^r |\mathrm{SL}_4(\mathbb{F}_{p_i})| 
\geq 2^{-r}\prod_{i=1}^r p_i \\
&\geq r^{\frac{r}{2}}.
\end{split}
\end{equation}
Note that in both \eqref{LowerIndexBoundEq}, and \eqref{LowerIndexBound2Eq} we have used the bound (valid for large values of $r$) on primorials which appears in \url{en.wikipedia.org/wiki/Primorial}. We thus have 
\begin{equation} \label{My48Eq}
\frac{\log([G : L_r])}{\log\log([G : L_r])} \overset{\ref{LowerIndexBoundEq}}{\underset{\ref{LowerIndexBound2Eq}}{\leq}} 
\frac{48r\log(2r)}{\log(\frac{r}{2}) + \log\log(r)} \leq
\frac{48r\log(2r)}{\log(2r)} = 48r.
\end{equation}
Setting $n_r = [G : L_r]$, we see that
\begin{equation}
d(L_r) \stackrel{\ref{LowerRankBoundEq}}{\geq} r = \frac{1}{48}48r \stackrel{\ref{My48Eq}}{\geq} \frac{\log(n_r)}{48\log\log(n_r)}
\end{equation} 
giving us an asymptotic growth of the rank as a function of the index which is only slightly slower than the upper bound \eqref{NormalBoundEq}.

\section*{Acknowledgments}

I would like to sincerely thank Miklos Abert, Andrei Jaikin-Zapirain, and Nikolay Nikolov for many discussions about this work.

Mark Shusterman\\

markshus@mail.tau.ac.il

\end{document}